\documentclass[11pt]{amsart}%
\usepackage{amssymb}
\usepackage{amsmath}
\usepackage{graphicx}
\usepackage[latin1]{inputenc}
\usepackage{amsfonts}
\usepackage[spanish,english]{babel}
\usepackage{comma}%
\providecommand{\U}[1]{\protect\rule{.1in}{.1in}}
\providecommand{\U}[1]{\protect\rule{.1in}{.1in}}
\newtheorem{theorem}{Theorem}
\theoremstyle{plain}

\newtheorem{definition}[theorem]{Definition}

\numberwithin{equation}{section}
\begin{document}

\title[\textbf{$q$-pseudoprimality }]
{\textbf{$q$-PSEUDOPRIMALITY:  A NATURAL GENERALIZATION OF STRONG PSEUDOPRIMALITY}}
\author[J.H. Castillo]{John H. Castillo}
\address{John H. Castillo, Departamento de Matemáticas y Estadística, Universidad de Nariño}
\email{jhcastillo@gmail.com, jhcastillo@udenar.edu.co}
\author[G. García-Pulgarín]{Gilberto Garc\'\i a-Pulgar\'in}
\address{Gilberto Garc\'\i a-Pulgar\'in, Universidad de Antioquia}
\email{gilberto.garcia@udea.edu.co}
\author[J.M Velásquez Soto]{Juan Miguel Vel\'asquez-Soto}
\address{Juan Miguel Vel\'asquez Soto, Departamento de Matemáticas, Universidad del Valle}
\email{juan.m.velasquez@correounivalle.edu.co}
\date{}

\begin{abstract}
In this work we present a natural generalization of strong pseudoprime to base $b$, which we have called $q$-pseudoprime to base $b$. It allows us to present another way to define a Midy's number to base $b$ (overpseudoprime to base $b$). Besides, we count the bases $b$ such that $N$ is a $q$-probable prime base $b$ and those ones such that $N$ is a Midy's number to base $b$. Furthemore, we prove that there is not a concept analogous to Carmichael numbers to $q$-probable prime to base $b$ as with the concept of strong pseudoprimes to base $b$.

\end{abstract}

\begingroup
\def\uppercasenonmath#1{} 
\maketitle
\endgroup

\section{Introduction}

Recently, Grau et al. \cite{grau} gave a generalization of Pocklignton's Theorem (also known as Proth's Theorem) and   Miller-Rabin primality test, it takes as reference some works of  
Berrizbeitia, \cite{berrizbeitia,BO2008}, where it is presented an extension to the concept of strong pseudoprime, called $\omega$-primes. As  Grau et al. said it is right, but its application is not too good because it is needed  $m$-th primitive roots of unity, see \cite{grau,zhang3}.

\

 In \cite{grau}, it is defined when an integer $N$ is a $p$-strong probable 
prime base $a$, for $p$  a prime divisor of $N-1$ and $\gcd(a,N)=1$. 
In a reading of that paper, we discovered that if a number $N$ is a $p$-strong probable prime to base $2$ for each $p$ prime divisor of $N-1$, it is actually a Midy's number or a overpseudoprime number  to base $2$. For instance, they said that  $2047$, $3277$, $4033$, $8321$, $65281$, $80581$, $85489$ and $88357$ are the first 
$p$-strong pseudoprimes to base $2$ for any prime $p\mid N-1$. Indeed, these integers are Midy's numbers to base $2$ and the first terms of the sequence $A141232$ at the Online Encyclopedia of Integers Sequences, OEIS, where 
we called them overpseudoprimes to base $2$.

It is important to highlight that the definition of $p$-strong probable prime to base $a$, does not require each divisor of $N-1$ to be congruent with $1$ module $p$. And this fact gives some difficulties as those ones showed in our Theorem \ref{MA}.

We organize this paper as follows. In the second section, we present the definition of  Midy's property and recall some known results about it, in particular we recall the Midy's number concept, some of its properties and
some connections between it and other former concepts of
pseudoprimes. \ We study properties of the set of integers $b$ such that 
 $N$ is a Midy's number to base $b$ and as new result we count the number of them.  

 In the third section we recall our concept of $q$-pseudoprimality to base $b$
and we stablish a formula that gives the number of bases of $q$-pseudoprimality. 

Finally, in the fourth section, we make some comments about the recent 
paper of Grau et al. \cite{grau}, where in an independent way are presented some of
our ideas.

\section{Midy's Property and Midy's numbers}

Let $N$ and $b$ be positive integers relatively primes, $b>1$ the base of numeration, 
$|b|_N$ the order of $b$ in the  multiplicative group $\mathbb{U}_N$ of positive 
integers less than $N$ and relatively primes with $N$, and $x\in \mathbb{U}_N$.
It is well known that when we write the fraction $\frac{x}{N}$ in base $b$, it is periodic. 
By period we mean the smallest repeating sequence of digits in base $b$ in such expansion,
it is easy to see that $\left\vert b\right\vert _{N}$ is the length of the
period of the fraction $\frac{x}{N}$ (see Exercise 2.5.9 in \cite{Nathanson}%
). Let $d,\,k$ be positive integers with $\left\vert b\right\vert _{N}=dk$,
$d>1$ and $\frac{x}{N}=0.\overline{a_{1}a_{2}\cdots a_{\left\vert b\right\vert
_{N}}}$ where the bar indicate the period and $a_{i}$'s are digits in base
$b$. We separate the period ${a_{1}a_{2}\cdots a_{\left\vert b\right\vert
_{N}}}$ in $d$ blocks of length $k$ and let
\[
A_{j}=[a_{(j-1)k+1}a_{(j-1)k+2}\cdots a_{jk}]_{b}%
\]
be the number represented in base $b$ by the $j$-th block and $S_{d}%
(x)=\sum\limits_{j=1}^{d}A_{j}$. If for all $x\in\mathbb{U}_{N}$, the sum
$S_{d}(x)$ is a multiple of $b^{k}-1$ we say that $N$ has the Midy's property
for $b$ and $d$. It is named after E. Midy (1836), to read historical aspects
about this property see \cite{Lewittes} and its references.

We denote with $\mathcal{M}_{b}(N)$ the set of positive integers $d$ divisors of
$|b|_{N}$ such that $N$ has the Midy's property for $b$ and $d$ and we will
call it the Midy's set of $N$ to base $b$. As usual, let $\nu_{p}(N)$ be the
greatest exponent of $p$ in the prime factorization of $N$ and $\omega(N)$
denotes the number of prime divisors of $N$.

For example $13$ has the Midy's property to the base $10$ and $d=3$, because
$|13|_{10}=6$, $1/13=0.\overline{076923}$ and $07+69+23=99$. Also, $49$ has
the Midy's property to the base $10$ and $d=14$, since $|49|_{10}=42$,
$1/49=0.\overline{020408163265306122448979591836734693877551}$ and
$020+408+163+265+306+122+448+979+591+836+734+693+877+551=7*999$. But $49$ does
not have the Midy's property to $10$ and $7$. Actually, we can see that
$\mathcal{M}_{10}(13)=\{2,3,6\}$ and $\mathcal{M}_{10}(49)=\{2, 3, 6, 14, 21,
42\}$.

In \cite{garcia09} was given the following characterization of Midy's property.

\begin{theorem}
\label{ppl2} If $N$ is a positive integer and $\left\vert b\right\vert
_{N}=kd$, then $d\in\mathcal{M}_{b}(N)$ if and only if $\nu_{p}(N)\leq\nu
_{p}(d)$ for all prime divisor $p$ of $\gcd(b^{k}-1,\ N)$.
\end{theorem}

Last theorem stablishes that if $d$ is a divisor of $\left\vert
b\right\vert _{N}$ then $d\in\mathcal{M}_{b}(N)$ if and only if for each prime
divisor $p$ of $N$ such that $\nu_{p}\left(  N\right)  >\nu_{p}\left(
d\right)  $, there exists a prime $q$ divisor of $\left\vert b\right\vert
_{N}$ that satisfies $\nu_{q}\left(  \left\vert b\right\vert _{p}\right)
>\nu_{q}\left(  \left\vert b\right\vert _{N}\right)  -\nu_{q}\left(  d\right)
$.

We demonstrated, see \cite[Cor. 1]{trio}, that if $d_{1}\in\mathcal{M}_{b}(N)$ and
$d_{2}$ is a divisor of $\left\vert b\right\vert _{N}$ and $d_{1}\mid d_{2}$
then $d_{2}\in\mathcal{M}_{b}(N)$, in this way the set $\mathcal{M}_{b}(N)$ is
closed \textquotedblleft for multiples\textquotedblright.

In \cite[Th. 2.4]{integracion}, we proved the next result.

\begin{theorem}
Let $N,q,v$ be integers with $q$ prime and $v>0$. Then $q^{v}\in
\mathcal{M}_{b}(N)$ if and only if $N=q^{n}p_{1}^{h_{1}}p_{2}^{h_{2}}\cdots
p_{l}^{h_{l}}$ where $n$ is a non-negative integer, $p_{i}$'s are different
primes and $h_{i}$'s are non-negatives integers not all zero, verifying $0\leq
n\leq v$, $\nu_{q}(\left\vert b\right\vert _{p_{i}})>0$ and
\[
\nu_{q}(\left\vert b\right\vert _{N})-v<\min\limits_{1\leq i\leq l}\left\{
\nu_{q}(\left\vert b\right\vert _{p_{i}})\right\}  \text{.}%
\]

\end{theorem}

Observe that when $N$ is a prime number, Theorem \ref{ppl2} is satisfied for any 
base $b$ and each prime divisor of $|b|_N$ and therefore it is also verified for 
any divisor greater than $1$ of $|b|_N$. Composite numbers with this property for 
a fixed base $b$ were studied in \cite{integracion} and \cite{GGP_shevelev}, under 
the name of Midy numbers to base $b$ or overpseudoprimes to base $b$.

\begin{definition}
We say that a number $N$ is a Midy's number to base $b$ (or overpseudoprime to
base $b$) if $N$ is an odd composite number relatively prime to both $b$ and
$\left\vert b\right\vert _{N}$ and for all divisor $d>1$ \ of $\left\vert
b\right\vert _{N}$ we get that $d\in\mathcal{M}_{b}(N)$.
\end{definition}

It is easy to see, from \cite[Th. 2.10]{integracion} or \cite[Th.
12]{GGP_shevelev}, that an odd composite number $N$ with $N$ relatively prime
with $b$, is a Midy's number to base $b$ if and only if $\left\vert
b\right\vert _{N}=\left\vert b\right\vert _{p}$ for every prime $p$ divisor of
$N$. Thereby, we have that an odd composite $N$ is a Midy's number to base $b$
if and only if each divisor of $N$ is either a prime or a Midy's number to
base $b$. Observe that $|1|_N=1$ for any positive integer $N$, for that reason 
for now on we accept $b$ to be equal to $1$, although it does not make sense 
to consider the Midy's property to the base $1$.

The result below, Theorem 2.3 of \cite{Motose2}, allows us to give another
characterization of Midy's numbers.

\begin{theorem}
[Theorem 2.3 of \cite{Motose2}]\label{motose}Let $m$, $b\geq2$, $n\geq3$ and
$r$ be integers, where $r$ is the greatest prime divisor of $n$. Then $m$ is a
divisor of $\Phi_{n}\left(  b\right)  $ if and only if $b^{n}\equiv
1\ \operatorname{mod}\ m$ and every prime divisor $p$ of $m$ satisfies that
\[
n=\left\{
\begin{array}
[c]{ccc}%
\left\vert b\right\vert _{p} &  & \text{if }r\neq p,\\
&  & \\
r^{e}\left\vert b\right\vert _{r} &  & \text{if }r=p.
\end{array}
\right.
\]

\end{theorem}

We denote with $\Phi_{n}\left(  x\right)  $ the $n$-th cyclotomic polynomial
with rational coefficients. \ From the above theorem we get immediatly the
next characterization of Midy's numbers.

\begin{theorem}
A composite number $N$ with $\gcd\left(  N,\left\vert b\right\vert
_{N}\right)  =1,$ is a Midy's number to base $b$ if and only if $\Phi
_{\left\vert b\right\vert _{N}}\left(  b\right)  \equiv0\ \operatorname{mod}%
\ N$.
\end{theorem}

Theorem \ref{motose} allow us to conclude, in particular, that if $n$ and $b$
are integers with $n\geq3$ and $b\geq2$ then, $\gcd\left(n,\Phi_{n}\left(
b\right)  \right)  $ is either $1$ or the greatest prime divisor of $n$,
therefore this result give us a way to produce Midy's numbers. Indeed, if
$p$ is a prime divisor of $\frac{\Phi_{N}\left(  b\right)  }{\gcd\left(
N,\ \Phi_{N}\left(  b\right)  \right)  }$,  then $\left\vert b\right\vert
_{p}=N$ and so we get the following result.

\begin{theorem}
Let $N>2$ and $f_{N}\left(  b\right)  =\frac{\Phi_{N}\left(  b\right)  }%
{\gcd\left(  N,\ \Phi_{N}\left(  b\right)  \right)  }$. If $f_{N}\left(
b\right)  $ is composite, then $f_{N}\left(  b\right)  $ is a Midy's number to
base $b$.
\end{theorem}

With the help of this generating way of Midy's numbers, it is easy to prove
that if $b$ is an even number then the generalized Fermat number $b^{2^{n}}+1$
is either a prime or a Midy's number; the same happens with the generalized
Mersenne numbers $\frac{b^{p}-1}{b-1}$ where $p$ is prime which not divides
$b-1$. In this way Fermat and Mersenne were not totally wrong about the
primality of their numbers.

Now, we study some connections between Midy's numbers and some kind of
pseudoprimes. The composite integer $N$ is called a pseudoprime (or Fermat
pseudoprime) to base $b$ if $\gcd\left(  b,N\right)  =1$ and $b^{N-1}%
\equiv1\ \operatorname{mod}\ N$. An integer which is pseudoprime for all
possible bases $b$ is called a Carmichael number or an absolute pseudoprime.
An odd composite $N$ such that $N-1=2^{s}t$ with $t$ an odd integer and
$\gcd\left(  b,\ N\right)  =1$, is said to be a strong pseudoprime to base
$b$~ if either $b^{t}\equiv1\ \operatorname{mod}\ N$ or $b^{2^{i}t}%
\equiv-1\ \operatorname{mod}\ N$, for some $0\leq i<s$. \ It can be prove that
an odd composite integer $N$ is a strong pseudoprime to base $b$ if and only
if $N$ is pseudoprime to base $b$ and there is a non-negative integer $k$ such
that $\nu_{2}\left(  \left\vert b\right\vert _{N}\right)  =\nu_{2}\left(
\left\vert b\right\vert _{p}\right)  =k$ for all prime $p$ divisor of $N$.

 The set of bases of Midy pseudoprimality is closed under taking powers,
although it is not closed under multiplication as we can see when take $N=91$ 
which is Midy's number to bases $9$ and $16$ but it is not to $53$, their product
modulo $N$.

\begin{theorem}
\label{midypseudoprimeisstrong} If $N$ is a Midy's number to base $b$, then $N$
is a strong pseudoprime to base $b$.
\end{theorem}

\begin{proof}
Since $N$ is a Midy's number to base $b$ implies that $\left\vert b\right\vert
_{N}=\left\vert b\right\vert _{n}$ for each divisor $n$ of $N$ and thus there
is a non-negative integer $k$ such that for all prime divisor $p$ of $N$ we
get that $\nu_{2}\left(  \left\vert b\right\vert _{p^{\nu_{p}\left(  N\right)
}}\right)  =k$ and the result follows.
\end{proof}

The reciprocal is not true. For example $N=91$ is a strong pseudoprime to base
$53$, but it is not a Midy's number to this base. \ Additionally, from the
last theorem we could say that the Midy's numbers are ``stronger'' than
strong pseudoprimes.

We finish this section counting the number of positive integers $b$ such that $N$ is a
Midy pseudoprime to base $b$.

\begin{theorem}
Let $N=\prod\limits_{1\leq i\leq\omega(N)}p_{i}^{e_{i}}$ be an integer where
$p_{i}$ are different primes and $D=\gcd\left(  p_{1}-1,\ p_{2}-1,\ \ldots
,\ p_{\omega(N)}-1\right)  $, then the number of elements $b\in\mathbb{U}_{N}$
such that $N$ is a Midy's number base $b$ is given by
\begin{equation}
B_{m}(N)=\sum\limits_{d\mid D}\phi(d)^{\omega(N)}. \label{uno}%
\end{equation}

\end{theorem}

\begin{proof}
Take $b$ such that $N$ is a Midy's number base $b$. From \cite[Theorem
12]{GGP_shevelev} we know that $\left\vert b\right\vert _{p_{i}}=\left\vert b\right\vert
_{p_{1}}$ for each $i=2,\ldots,\omega(N)$ and for that reason $\left\vert
b\right\vert _{p_{1}}$ is a divisor of $D$. \ On the other hand, if $d$ is a
divisor of $D$, for each $i$ there are $\phi\left(  d\right)  $ elements
$b_{i}\in\mathbb{U}_{p_{i}}$ of order $d.$ \ The Chinese Remainder Theorem
allow us to obtain an element $b$ such that $b\equiv b_{i}\ \operatorname{mod}%
\ p_{i}$ where $1 \leq i\leq\omega(N)$ and thus if we take all these elements
$b_{i}$, we get $\phi(d)^{\omega(N)}$ possible bases $b$ and the result follows.
\end{proof}

\section{$q$-pseudoprimality: \ A natural generalization of strong
pseudoprimality.}

If $N$ is a Midy's number to base $b$, we have showed that $\nu_{q}\left(
\left\vert b\right\vert _{N}\right)  =\nu_{q}\left(  \left\vert b\right\vert
_{p}\right)  $ for all primes $q$ and $p$, with $p\mid N$. \ This last fact and the
characterization of strong pseudoprimality suggest the following definition:

\begin{definition}
\label{qprim}Let $N$, $b$ be integers with $\gcd\left(  b,N\right)  =1$ and
b$^{N-1}\equiv1\ \pmod{N}$ and $q$ a prime number such that for every prime
$p$ divisor of $N$, $q$ divides $p-1$. We say that $N$ is a $q$-probable prime
base $b$ if there is a non negative integer $k$ such that for every prime $p$
divisor of $N$ we have $\nu_{q}\left(  \left\vert b\right\vert _{p}\right)
=k$. Morever, if $N$ is composite we say that $N$ is $q$-pseudoprime to base
$b$.
\end{definition}

Even though, in the above definition it is necessary to calculate $|b|_p$, for 
each prime $p$ divisor of $N$, and to verify that $q$ appears the same number 
of times in each one of these numbers, actually we can decide the $q$-probable 
primality of a given number with a procedure similar to the Miller's Test. 
We proved it in \cite{integracion} and we present it in the next theorem.

\begin{theorem}[Theorem 3.6 of \cite{integracion}]\label{q-seudo}Assume that $N$ is and odd
integer, $b$ a positive integer relatively prime with $N$ and $q$ a prime that
divides $p-1$ for all prime divisor $p$ of $N$. Write $N-1=q^{s}t$ with
$\gcd\left(  q,t\right)  =1$, then $N$ is a $q$-probable prime base $b$ if and
only if one of the following conditions holds:

\begin{enumerate}
\item $b^{t}\equiv1\pmod N$.

\item There exists with $0\leq i<s$ such that $N$ divides $\Phi_{q}\left(
b^{q^{i}t}\right)  $.
\end{enumerate}

Furthermore, if the second condition holds, then every prime divisor of $N$ is
congruent with $1$ modulo $q^{i+1}$.
\end{theorem}

Grau et al. \cite{grau} defined the concept of $q$-strong
probable prime to base $b$. Their idea is similar to our Definition
\ref{qprim}, although they do not require $q$ to be a divisor of $p-1$ for
each prime $p$ divisor of $N$ and this could leave to inconvenient facts as we
show in the following result.

\begin{theorem}[Theorem 3.9 of \cite{integracion}]\label{MA} Let $N$ be a Carmichael number and $q$ a
prime with $N-1=q^{s}t$ where $q$ does not divide $t$ and such that, for each
prime divisor $p$ of $N$, $q$ does not divide $p-1$, then $b^{t}%
\equiv1\pmod N$, for each integer $b$ relatively prime with $N$.
\end{theorem}

In consequence, if we removed the condition that each prime factor of $N$ to 
be of the form $hq+1$, one would define an analogue concept to Carmichael number
from the concept of $q$-probable prime.

For instance, this is the case of the Carmichael number:
\[
N=2\,333\,379\,336\,546\,216\,408\,131\,111\,533\,710\,540\,349\,903\,201
\]
which is product of $23$ primes, as follows
\begin{align*}
N=  &  11\times13\times17\times19\times29\times31\times37\times41\times
43\times47\times61\times71\times\\
&  \times73\times101\times109\times113\times127\times139\times163\times
211\times337\times421\times541\text{.}%
\end{align*}
Taking $q=12\,068\,159$ which is prime and $N-1=qt$ where%
\[
t=193\,350\,065\,784\,368\,304\,074\,474\,949\,634\,864\,800
\]
and if $b$ is relatively prime with $N$ then $b^{t}\equiv1\pmod N$. For that reason, in the
concept of Grau et al. \cite{grau}, it is a $q$-strong probable
prime, while it does not make sense to talk  about the $q$-probable primality of $N$.

We denote the number of bases of probable primality of $N$ with $B_{pp}\left(
N\right)  $ and its number of bases of strong probable primality with
$B_{spp}\left(  N\right)  $. \ \ It is well known that, see \cite[Exercises
3.14 and 3.15]{prime}:%

\begin{equation}
B_{pp}\left(  N\right)  =\prod\limits_{p\mid N}\gcd\left(  p-1,\ N-1\right)
\label{tres}%
\end{equation}

\begin{equation}
B_{spp}\left(  N\right)  =\left(  1+\frac{2^{\nu(2,N)\omega(N)}-1}%
{2^{\omega(N)}-1}\right)  \prod\limits_{p\mid N}\gcd\left(  p-1,\ t\right),
\label{dos}%
\end{equation}

where $N-1=2^{s}t$ with $t$ an odd number, $\omega(N)$ the number of prime divisors of
$N$ and for a prime $q$ we write $$\nu(q,N)=\nu_{q}\left(  \gcd\left(  p_{1}-1,
p_{2}-1, \ldots, p_{\omega(N)}-1\right)  \right).$$

Now we will count the number of bases $b$ such that $N$ is a $q$-probable
prime base $b$ and we denote it with $B_{qpp}\left(  N\right)  $.

\begin{theorem}
\label{basesqpp}Suppose that $N$ is an odd integer and $q$ a prime number such
that each prime divisor of $N$ is congruent with $1$ modulo $q$. Assume that
$N-1=q^{s}t$ where $q$ and $t$ are relatively primes, then
\[
B_{qpp}\left(  N\right)  =\left(  1+\left(  q-1\right)  ^{\omega(N)}%
\frac{q^{\nu\left(  q,N\right)  \omega(N)}-1}{q^{\omega(N)}-1}\right)
\prod_{p\mid N}\gcd\left(  \ p-1,\ t\right)  .
\]

\end{theorem}

\begin{proof}
Suppose that $N$ is a $q$-probable prime base $b$. By Theorem \ref{q-seudo}
we will consider two cases. Firstly, we assume that there exists an integer
$i$ such that $0\leq i<\nu\left(  q,N\right)  $, $\operatorname*{gcd}%
\left(  b^{q^{i}t}-1,\ N\right)  =1$ and $b^{q^{i+1}t}\equiv
1\operatorname{mod}N$. Let $p$ be a prime divisor of $N$. Thus, we get that
$\operatorname*{gcd}\left(  b^{q^{i}t}-1,\ p^{\nu_{p}\left(  N\right)
}\right)  =1$ and $b^{q^{i+1}t}\equiv1\operatorname{mod}p^{\nu_{p}\left(
N\right)  }$. Taking $m=q^{i}t$ it follows that $b^{qm}\equiv
1\operatorname{mod}p^{\nu_{p}\left(  N\right)  }$ and $b^{m}\not \equiv
1\operatorname{mod}p$, therefore the number of bases $b$, denoted by $h_{p}$,
is
\begin{align*}
h_{p}  &  =\operatorname*{gcd}\left(  qm,\ \phi(p^{\nu_{p}\left(  N\right)
})\right)  -\operatorname*{gcd}\left(  m,\ p-1\right) \\
&  =q^{i+1}\operatorname*{gcd}\left(  t,\ p-1\right)  -q^{i}%
\operatorname*{gcd}\left(  t,\ p-1\right) \\
&  =\left(  q-1\right)  q^{i}\operatorname*{gcd}\left(  \ p-1,\ t\right)
\text{.}%
\end{align*}

By the Chinese Remainder Theorem, we get that the number of solutions of
$\operatorname*{gcd}\left(  b^{q^{i}t}-1,\ N\right)  =1$ and $b^{q^{i+1}%
t}\equiv1\operatorname{mod}N$, is the product of $h_{p}$ when $p$ is a prime
divisor of $N$. Thus the number of these solutions is
\[
\left(  q-1\right)  ^{\omega(N)}q^{i\omega(N)}\prod_{p\mid N}%
\operatorname*{gcd}\left(  \ p-1,t\right)  \text{.}%
\]

Because $i$ takes values from $0$ until $\nu\left(  q,N\right)  -1$, then the
number of bases $b$ which satisfied the condition (2) of the Theorem
\ref{q-seudo} is
\[
\left(  q-1\right)  ^{\omega(N)}\prod_{p\mid N}\operatorname*{gcd}\left(
p-1,\ t\right)  \sum_{i=o}^{\nu\left(  q,N\right)  -1}q^{i\omega(N)};
\]
which implies
\begin{equation}
\left(  q-1\right)  ^{\omega(N)}\frac{q^{\nu\left(  q,N\right)  \omega(N)}%
-1}{q^{\omega(N)}-1}\prod_{p\mid N}\operatorname*{gcd}\left(  p-1,\ t\right)
\text{.} \label{popo}%
\end{equation}

Similarly, we count the number of bases $b$ verifying the first condition of
the Theorem \ref{q-seudo}, i.e. $b^{t}\equiv1\operatorname{mod}N$. \ This
number is equals to
\begin{equation}
\prod_{p\mid N}\operatorname*{gcd}\left(  p-1,\ t\right)  . \label{pipi}%
\end{equation}
\ The statement of the theorem follows when we add \eqref{popo} and \eqref{pipi}.
\end{proof}

\section{Acknowledgments}
The authors are members of the research group: Álgebra, Teoría de Números
y Aplicaciones, ERM. The results of this article are part of the research project ``Construcciones de conjuntos $B_h[g]$, propiedad  de Midy, y algunas aplicaciones'' CÓDIGO: 110356935047 partially financied by COLCIENCIAS.

\bibliographystyle{amsplain}
\bibliography{bibliografiaggp}

\end{document}